\documentclass{lmcs}

\usepackage{enumerate}
\usepackage{hyperref}
\usepackage{amssymb}
\usepackage{graphicx,epsfig,epstopdf}

\theoremstyle{plain}\newtheorem{remark}[thm]{Remark}
\theoremstyle{plain}\newtheorem{example}[thm]{Example}
\theoremstyle{plain}\newtheorem{definition}[thm]{Definition}
\theoremstyle{plain}

\newcommand{\lra}{\longrightarrow}

\newcommand{\ur}{\uparrow}
\newcommand{\dr}{\downarrow}

\begin{document}

\title[Faithfulness of directed
complete posets]{Faithfulness of directed
complete posets based on Scott closed
set lattices}
\author[D. Zhao]{Zhao Dongsheng}	
\address{National Institute of Education, Nanyang Technological University, 1 Nanyang Walk, Singapore 637616}	
\email{dongsheng.zhao@nie.edu.sg}  
\author[L. Xu]{Xu Luoshan}	
\address{Department of Mathematics, Yangzhou University, Yangzhou 225002, China}	
\email{luoshanxu@hotmail.com}  





\begin{abstract}
  \noindent By Thron, a topological space $X$ has the property that $C(X)$ isomorphic to $C(Y)$ implies $X$ is homeomorphic to $Y$ iff $X$ is sober and $T_D$, where $C(X)$ and $C(Y)$ denote the lattices of closed sets of $X$ and $T_0$ space $Y$, respectively. When we consider dcpos (directed complete posets) equipped their Scott topologies, a similar question arises: which dcpos $P$ have the property that for any dcpo $Q$, $C_{\sigma}(P)$ isomorphic to $C_{\sigma}(Q)$ implies $P$ is isomorphic to $Q$ (such a dcpo $P$ will be called Scott closed set lattice faithful,  or SCL-faithful in short)? Here $C_{\sigma}(P)$ and $C_{\sigma}(Q)$ denote the lattices of Scott closed sets of $P$ and $Q$, respectively. Following a characterization of continuous (quasicontinuous) dcpos in terms of $C_{\sigma}(P)$, one easily deduces that every continuous (quasicontinuous) dcpo is SCL-faithful. Note that  the Scott space of every  continuous (quasicontinuous) dcpo is sober. Compared  with  Thron's result, one naturally asks whether every SCL-faithful dcpo is sober (with the Scott topology). In this paper we shall prove that some classes of dcpos are SCL-faithful, these classes contain some  dcpos whose Scott topologies are not bounded sober. These results will help to obtain a complete characterization of SCL-faithful dcpos in the future.
\end{abstract}

\maketitle


\section{Introduction}
In \cite{thron-1962}, Thron proved the  interesting result: a topological space $X$ has the property that $C(X)$ isomorphic to $C(Y)$ implies $X$ is homeomorphic to $Y$ iff $X$ is sober and $T_D$, where $C(X)$ and $C(Y)$ denote the lattices of closed sets of  $X$ and $T_0$ space $Y$, respectively.
A directed complete poset (dcpo, for short)  $P$ will be called Scott closed set lattice faithful, or SCL-faithful in short if for any dcpo $Q$, $P$ is isomorphic to $Q$ whenever the Scott-closed-set lattice $C_{\sigma}(P)$  of $P$ and $C_{\sigma}(Q)$ of $Q$ are isomorphic. One of the classic result in domain theory is that a dcpo $P$ is continuous iff the lattice $C_{\sigma}(P)$  is a completely distributive lattice (Theorem II-1.14 of \cite{comp}). From this it follows  that every continuous dcpo is SCL-faithful. In a similar way, one deduces that every quasicontinuous dcpo is SCL-faithful.
Compared  with  Thron's result, one naturally asks whether every SCL-faithful dcpo is sober in their Scott topology.

In \cite{johnstone-81}, Johnstone constructed the first dcpo whose Scott topology is not sober. Later Isbell \cite{isbell-1982a} constructed a complete lattice whose Scott topology is not sober and Kou \cite{kouhui} gave a dcpo whose Scott topology is well-filtered but not sober. In this paper, we will prove that some classes of dcpos, including all quasicontinuous dcpos as well as Johnstone's and Kou's examples, are SCL-faithful. The full characterization of all SCL-faithfull dcpos is still open.\\

\section{Preliminaries}

For any subset $A$ of a poset $P$, let $\ur\!\!A=\{x\in\!P: y\le\!x \mbox{ for some }y\in\!A\}$ and $\downarrow\!A =\{x\in\!P: x\le\!y \mbox{ for some } y\in\!A\}$. A subset $A$ is called an upper set if $A=\ur\!A$, and a lower set if $A=\downarrow\!A$.
A subset $U$ of a poset $P$ is Scott open  if (i) $U=\ur\!U$  and (ii) for
any directed subset $D$, $\bigvee D\in U$ implies $D\cap U\not=\emptyset$, whenever $\bigvee D$ exists.
All Scott open sets of a poset $P$ form a topology on $P$, denoted by $\sigma(P)$ and called the Scott topology on $P$. The complements of Scott open sets are called Scott closed sets. Clearly, a subset $A$ is Scott closed iff (i) $A=\dr A$ and (ii) for any directed subset $D\subseteq A$, $\bigvee D\in A$ whenever $\bigvee\!D$ exists. The set of all Scott closed sets of $P$ will be denoted by $C_{\sigma}(P)$. The space $(P, \sigma(P))$ is denoted by $\Sigma\!P$, called the Scott space of $P$ (See \cite{comp} for more about Scott spaces).

A poset $P$ is directed complete if its every directed subset has a supremum. A directed complete poset is briefly called a dcpo.

A subset $A$ of a topological space is irreducible if $A\subseteq F_1\cup F_2$ with $F_1$ and $F_2$ closed, then  $A\subseteq F_1$ or $A\subseteq F_2$ holds. The set of all nonempty irreducible closed subsets of space $X$ will be denoted by $Irr(X)$.

For any $T_0$ topological space $(X, \tau)$, the specialization order $\le_{\tau}$ on $X$ is defined by $x\le_{\tau} y$ iff $x\in cl(\{y\})$ where $``cl(\cdot)"$ means taking closure.
\begin{remark}

\label{irreducible sets}

(1) For any topological space $X$, $(Irr(X), \subseteq )$ is a dcpo. If  $\mathcal{D}$ is a directed subset  of $Irr(X)$, the supremum of $\mathcal{D}$ in $(Irr(X), \subseteq )$ equals $cl(\bigcup\mathcal{D})$ (the closure of $\bigcup\mathcal{D}$), which is the same as the supremum of $\mathcal{D}$ in the complete lattice  of all closed sets of $X$.

(2) For any $x\in X$, $cl(\{x\})\in Irr(X)$. A $T_0$ space $X$ is called sober if $Irr(X)=\{cl(\{x\}): x\in X\}$, that is, every nonempty irreducible closed set is the closure of a point.

(3) If $(X,\tau)$ and $(Y,\eta)$ are topological spaces such that  the open set lattices ${(\tau,\subseteq)}$ and $(\eta,\subseteq)$ of $X$ and $Y$ are isomorphic, then the posets $Irr(X)$ and $Irr(Y)$ are isomorphic.

\end{remark}

For a $T_0$ space $X$, a sobrification of $X$ is a sober space $Y$ together with a continuous mapping $\eta_X: X\lra Y$, such that for any continuous mapping $f: X\lra Z$ with $Z$ sober, there is a unique continuous mapping $\hat{f}: Y\lra Z$ such that $f=\hat{f}\circ \eta_X$. The sobrification of a $T_0$ space is unique up to homeomorphism.

\begin{remark}
\label{soberfication}
The following facts on sobrifications are well-known.

(1) If $Y$ is a sober space, then $Y$ is a sobrification  of a $T_0$ space $X$  iff  the  closed  set lattice $C(X)$ of $X$  is isomorphic to the closed set lattice $C(Y)$ of $Y$ (Equivalently, the open set lattice of $Y$ is isomorphic to that of $X$).

(2) The set $Irr(X)$ of all nonempty closed irreducible sets of a $T_0$ space $X$ equipped the hull-kernel topology is a sobrification  of $X$, where the mapping $\eta_X: X\lra Irr(X)$ is defined by $\eta_X(x)=cl(\{x\})$ for all $x\in X$.
 The closed sets of the hull-kernel topology consists of all sets of the form $h(A)=\{F\in Irr(X): F\subseteq A\}$ ($A$ is a closed set of $X$).
\end{remark}

A  $T_0$ space will be called  Scott sobrificable if there is a dcpo $P$ such that $\Sigma\!P$ is the sobrification  of $X$. Also for a $T_0$ space $(X, \tau)$, $(X, \tau)$ is homeomorphic to $\Sigma\,P$ for some poset $P$ iff $(X, \tau)$ is homeomorphic to the Scott space $\Sigma(X, \le_{\tau})$.

\begin{lem}\label{scott soberficable}
A $T_0$ space $(X, \tau)$ is Scott sobrificable iff
for any Scott closed set  $\mathcal{F}$ of the dcpo $Irr(X)$,
there is a closed set $A$ of $X$ such that $\mathcal{F}=h(A)$.
\end{lem}
\begin{proof} Note that $Irr(X)$ equipped with the hull-kernel topology is a sobrification of $X$. Also every $h(A)=\{F\in Irr(X): F\subseteq A\}$ is a Scott closed set of the dcpo $(Irr(X), \subseteq)$, where $A$ is a closed set of $X$. Thus the hull kernel topology on $Irr(X)$ is contained in the Scott topology of $Irr(X)$.
Thus $(X, \tau)$ is Scott sobricable iff  there is a  dcpo $P$ such that $Irr(X)$ is homeomorphic to $\Sigma\!P$.
This is then equivalent to that the hull kernel topology on $Irr(X)$ coincides with the Scott topology on $(Irr(X), \subseteq)$, hence every  Scott closed set of $(Irr(X), \subseteq)$ equals $h(A)$ for some closed set $A$ of $(X, \tau)$.
\end{proof}

A topological space $(X, \tau)$ is called a d-space (or monotone convergence space) if (i) $X$ is $T_0$, (ii) the poset $(X, \le_{\tau})$ is a dcpo, and (iii) for any directed subset $D\subseteq X$, $D$ converges (as a net) to $\bigvee\!D$. If $(X, \tau)$ is a d-space, then every closed set $F$ of $X$ is a Scott closed set of the dcpo $(X, \leq_{\tau})$.

\begin{remark}\label{d-spaces}

(1) Every sober space is a d-space.

(2) Every Scott space of a dcpo is a d-space.

\end{remark}

\begin{lem}\label{sup of directed set of points}
Let $(X, \tau)$ be a d-space.
If $\{x_i: i\in I\}$ is a directed subset of $(X, \le_{\tau})$, then
the supremum $\sup\{cl(\{x_i\}): i\in I\}$ of $\{cl(\{x_i\}): i\in I\}$ in $Irr(X)$ equals
$cl(\{x\})$, where $x=\bigvee\{x_i: i\in I\}$.
\end{lem}

\section{Main results}
In this section, we  establish some classes of SCL-faithful dcpos, using irreducible sets, quasicontinuous elements and M property, respectively.

A $T_0$ space is called  bounded-sober  if every nonempty upper bounded (with respect to the specialization order on $X$) closed irreducible subset of the space is the closure of a point \cite{zandf-2010}. Every sober space is bounded-sober, the converse implication is not true.

\vskip 0.5cm
If $X$ is a $T_0$ space such that every  irreducible closed {\sl proper } subset is the closure of an element, then  $X$ is bounded-sober.
\vskip 0.5cm
In the following, a dcpo whose Scott space is sober (bounded-sober) will be simply called a sober (bounded-sober) dcpo.

\begin{lem}\label{non-sober bounded sober}
For a bounded-sober dcpo $P$, $\Sigma\!P$ is Scott sobrificable if and only if $P$ is sober.
\end{lem}
\begin{proof} We only need to check  that  if $\Sigma\!P$ is not sober, then it is not  Scott sobrificable.

Since $\Sigma\! P$ is not sober, there is a nonempty irreducible closed set $F$ such that $F$ is not the closure of any point.
 By that $\Sigma\!P$ is bounded-sober, one can verify that the set $\mathcal{F}=\downarrow_{Irr(\Sigma\,P)}\{cl(\{x\}): x\in F\}$ is a Scott closed set of $Irr(\Sigma\,P)$. But  any closed set $B$ of $\Sigma\,P$ containing all $cl(\{x\}) (x\in F)$ must contain $F$, thus $h(B)\not=\mathcal{F}$.  By Lemma \ref{scott soberficable},  $\Sigma\,P$ is not Scott sobrificable.
\end{proof}

In the following, we shall write $P\cong Q$ if the two posets $P$ and $Q$ are isomorphic.
\begin{thm}\label{tm-sober-bounded sober}
Let  $P$ be a sober dcpo. For any bounded-sober dcpo $Q$, if $C_{\sigma}(P)\cong C_{\sigma}(Q)$ then  $P \cong Q$.
\end{thm}
\begin{proof}
Let $Q$ be a bounded-sober dcpo such that $C_{\sigma}(P)\cong C_{\sigma}(Q)$. Then $\Sigma\! P$ is a sobrification  of $\Sigma\,Q$. By Lemma \ref{non-sober bounded sober}, $\Sigma\! Q$ is sober, therefore $\Sigma\!P$ and $\Sigma\! Q$ are homeomorphic, which implies $P\cong Q$.
\end{proof}

\begin{definition} An element $a$ of a poset $P$ is called down-linear if the subposet $\dr a=\{x\in P: x\le a\}$ is a chain (for any $x_1, x_2\in \dr a$, it holds that either $x_1\le x_2$ or $x_2\le x_1$).
\end{definition}

The image of a down-linear element under an order isomorphism is clearly down-linear.

\begin{lem}\label{irr set with liear down set}
Let $X$ be a d-space. If  $F\in Irr(X)$ is a down-linear element of the poset $Irr(X)$, then there exists a unique $x\in X$ such that $F=cl(\{x\})$.
\end{lem}
\begin{proof}
First, the set $\{cl(\{x\}): x\in F\}$ is a subset of $\dr F$ in $Irr(X)$, so it is a chain. Thus $\{x: x\in F\}$ is a chain of $(X, \le_{\tau})$. Let $x=\sup\{x: x\in F\}$. Then noticing that $F$ is closed, we have  $cl(\{x\})=F$.
\end{proof}

In the following, for a dcpo $P$, we shall use $Irr_{\sigma}(P)$ to denote the dcpo of all nonempty irreducible Scott closed subsets of $P$.
Without specification, irreducible sets of a poset mean the irreducible sets with respect to the Scott topology.

\begin{thm}\label{main theorem}
Let $P$ be a dcpo  satisfying the following conditions:

 {\em (DL-sup)} for any nonempty irreducible Scott closed \it{proper} set $F$, $F$  is either a down-linear element of $Irr_{\sigma}(P)$  or it is the supremum of a directed set of down-linear irreducible closed sets.\\
Then $P$ is SCL-faithful.
\end{thm}
\begin{proof}
Let dcpo $P$ satisfy the above condition (DL-sup) and $Q$ be a dcpo such that $C_{\sigma}(P)\cong C_{\sigma}(Q)$.

(1) Let $F\in Irr_{\sigma}(P)$ and $F\not=P$. If $F$ is down-linear,  then by Lemma \ref{irr set with liear down set}, $F$ is the closure of a unique point.
If $F$ is the supremum of a directed set of down-linear irreducible closed sets, then by Lemma \ref{irr set with liear down set},  $F$ can be represented as the directed supremum of $cl(\{x_i\}) (i\in J)$. Thus, $\{x_i: i\in J\}$ is a directed set of  $P$. Let $x=\sup\{x_i: i\in J\}$. Then noticing that $F$ is closed, we have that $cl(\{x\})=F$ is also a closure of a point.

(2) Since $C_{\sigma}(P)\cong C_{\sigma}(Q)$, $Q$ also satisfies condition  (DL-sup). As the proof of  (1) only make use of condition (DL-sup), so every nonempty closed irreducible proper subset of $\Sigma\,Q$ is  the closure of a point.

By the definition of the bounded-sobriety, we see that (1) and (2) imply that $\Sigma\,P$ and $\Sigma\,Q$ are all bounded-sober.

(3)  If either $\Sigma\!P$ or $\Sigma\!Q$ is sober, then by Theorem \ref{tm-sober-bounded sober}, $P\cong Q$.  If neither  $\Sigma\!P$ nor $\Sigma\!Q$  is sober, then $P$ and $Q$ are irreducible sets  and are not the closure of any point.  Thus $Q\cong\{cl(\{y\}): y\in Q\}\cong Irr_{\sigma}(Q)-\{Q\}\cong Irr_{\sigma}(P)-\{P\}\cong\{cl(\{x\}): x\in P\}\cong P$, as desired.
\end{proof}

\begin{example}\label{ex-JS-non-sober}
In \cite{johnstone-81}, Johnstone constructed the first non-sober dcpo  as $X=\mathbb{N}\times(\mathbb{N}\cup \{\infty\})$ with partial order defined by
$$ (m, n)\le (m', n') \Leftrightarrow \mbox{either } m=m' \mbox{ and } n\le n'\\
\mbox{or } n'=\infty \mbox{ and } n\le m'.$$
Then

(a) $(X, \leq)$ is a dcpo, $X$ is irreducible and $X\not=cl(\{x\})$ for any $x\in X$.

(b) If $F$ is a proper irreducible Scott closed set of $X$, then $F=\dr\!(m, n)$ for some $(m, n)\in X$.

(c) If $m\not=\infty$, $\dr\!(m, n)$ is a down-linear element of $Irr_{\sigma}(X)$. If $m=\infty$, then $\dr (m, n)$ is the supremum of the
chain $\{\dr\!(m, k): k \not=\infty\}$ whose members are down-linear.

Hence by Theorem \ref{main theorem}, we deduce that  dcpo $X=\mathbb{N}\times(\mathbb{N}\cup \{\infty\})$ is SCL-faithful.

Thus an SCL-faithful dcpo need not be sober.

\end{example}

Next, we give another  class of SCL-faithful dcpos.

\begin{remark} (cf. \cite{lawxu})
Let $A$ be a nonempty Scott closed set of a dcpo $P$. Then

(i) $A$ is a dcpo.

(ii) For any subset $B\subseteq A$, $B$ is a Scott closed set of dcpo $A$ iff it is a Scott closed set of $P$. Thus $C_{\sigma}(A)=\dr_{C_{\sigma}(P)}A=\{B\in C_{\sigma}(P): B\subseteq A\}$.

\end{remark}

A finite subset $F$ of a dcpo $P$ is way-below an element $a\in P$, denoted by $F\ll a $, if for any directed subset $D\subseteq P$, $a\le \bigvee D$ implies $D\cap\!\ur F\not=\emptyset$. A dcpo $P$ is quasicontinuous if for any $x\in P$, the family
$$ fin(x)=\{F: F \mbox{ is finite and } F\ll x\}$$
is a directed family (for any $F_1, F_2\in fin(x)$ there is $F\in fin(x)$ such that $F\subseteq \ur F_1\cap {\ur F_2}$)  and
for any $y\not\leq x $ there is $F\in fin(x)$ satisfying $y\not\in \ur F$ (see Definition III-3.2 of \cite{comp}). Every continuous dcpo is quasicontinuous.

Every quasicontinuous dcpo is sober (Proposition III-3.7 of \cite{comp}). A dcpo $P$ is quasicontinuous iff the Scott open set lattice of $P$ is hypercontinuous (Theorem VII-3.9 of \cite{comp}). From this and Remark 4, we have the following.

\begin{lem}\label{quasicontinuous is scl-faithul}
Every quasicontinuous dcpo is SCL-faithful.
\end{lem}
\smallskip

An element $x$ of a dcpo $P$ is called a quasicontinuous element if the sub-dcpo $\dr\!\! x$ is a quasicontinuous dcpo.

\begin{thm}\label{quasicontinuous elements}
 Let $P$ be a dcpo. Then $P$ is SCL-faithful if it satisfies the following two conditions:

 (1) $\Sigma P$ is bounded sober;

 (2) every element of $P$ is the supremum of a directed set of quasicontinuous elements.
\end{thm}
\begin{proof}
Assume that $P$ is a dcpo satisfying the two conditions. Let $Q$ be a dcpo and $F: C_{\sigma}(P)\lra C_{\sigma}(Q)$ be an isomorphism. Then $F$ restricts to an isomorphism $F: Irr_{\sigma}(P)\lra Irr_{\sigma}(Q)$.

(1) Let $x\in P$ be a quasicontinuous element. Then $F(\dr\!x)$ is in $C_{\sigma}(Q)$ and ${\dr_{C_{\sigma}(P)}(\dr\!x)}$ $=\{B\in C_{\sigma}(P): B\subseteq \dr x\}$ is isomorphic via $F$ to $\dr_{C_{\sigma}(Q)}F(\dr\!x)=\{E\in C_{\sigma}(Q): E\subseteq F(\dr x)\}=C_{\sigma}(F(\dr\!x))$ (all Scott closed sets of $F(\dr\!x)$). Since $\dr\!x$ is quasicontinuous, it is SCL-faithful. Hence $\dr\!x$ is isomorphic to $F(\dr x)$, implying that there is a largest element in $F(\dr\!x)$, denoted by $f(x)$. It is easily observable that the mapping $f$ is well defined on the set of quasicontinuous elements of $P$, and for any two quasicontinuous elements $x_1, x_2\in P$, $f(x_1)\le f(x_2)$ iff $x_1\le x_2$.

(2) If $x\in P$ is the supremum of a directed set $\{x_i: i\in I\}$ of quasicontinuous elements $x_i$, then $F(\downarrow\!x)= sup_{Irr_{\sigma}(Q)}\{F(\downarrow\!x_i): i\in I\}=sup_{Irr_{\sigma}(Q)}\{\downarrow\!f(x_i): i\in I\}=\downarrow\!y_x$, where $y_x=sup_{Q}\{f(x_i): i\in I\}$ and $f(x_i)$ is the element in $Q$ defined for quasicontinuous elements $x_i$ in (1). Let $f(x)=y_x$ again.

Thus we have an monotone mapping $f: P\longrightarrow Q$. Following that $F$ is an isomorphism, we have that $f(x_1)\ge f(x_2)$ iff $x_1\ge x_2$.

It remains  to show that $f$ is surjective.

(3) If $y\in\downarrow\!f(P)$, then $\downarrow\!y\subseteq F(\downarrow\!x)$ for some $x\in P$. Since $F$ restricts to an isomorphism between the dcpos $Irr_{\sigma}(P)$ and $Irr_{\sigma}(Q)$, there is $H\in Irr_{\sigma}(P)$ such that  $H\subseteq \downarrow\!x$ and  $F(H)=\downarrow\!y$.  But $P$ is bounded-sober, so $H=\downarrow\!x'$ for some $x'\in P$. It follows that $y=f(x')$, implying $y\in f(P)$. Therefore
 $f(P)$ is a down set of $Q$. Also clearly $f(P)$ is closed under sups of directed sets, so it is a Scott closed subset of $Q$.

 (4) Since $F$ is an isomorphism between the lattices $C_{\sigma}(P)$ and $C_{\sigma}(Q)$, $Q=F(P)=F(sup_{C_{\sigma}(P)}\{\downarrow\!x:  x\in P\})=sup_{C_{\sigma}(Q)}\{F(\downarrow\!x): x\in P\}=sup_{C_{\sigma}(Q)}\{\downarrow\!f(x): x\in P\}$.

  For each $x\in P$, $\downarrow\!f(x)\subseteq f(P)$ and $f(P)$ is a Scott closed set of $Q$, it holds then that $sup_{C_{\sigma}(Q)}\{\downarrow\!f(x): x\in P\}\subseteq f(P)$. Therefore $Q\subseteq f(P)$, which implies $Q=f(P)$. Hence  $f$ is also surjective. The proof is thus completed.
\end{proof}

If $x\in P$ is a down-linear element of a dcpo $P$, then $\dr\!x$ is a chain, so it is continuous (hence quasicontinuous).

\begin{cor}\label{tm-c1-c2-faithful}
If $P$ is a dcpo satisfying the following conditions, then $P$ is SCL-faithful:

(1) $P$ is bounded-sober.

(2) every element $a\in P$ is the supremum of a directed set of down-linear elements.
\end{cor}

\begin{example}\label{ex-Kou-non-sob}
In order to answer the question whether every well-filtered dcpo is sober posed by Heckmann\cite{heckmann},  Kou \cite{kouhui} constructed another non-sober dcpo $P$ as follows:

Let  $X=\{x\in\mathbb{R}: 0 < x \le 1\}$,  $P_0=\{(k, a, b)\in \mathbb{R}: 0 < k <1,  0 < b \le a \leq 1\}$  and
$$P=X\cup P_0.$$
Define the partial order $\sqsubseteq$ on $P$ as follows:

(i)  for $x_1, x_2\in X$, $x_1\sqsubseteq x_2$ iff $x_1=x_2$;

(ii) $(k_1, a_1, b_1) \sqsubseteq (k_2, a_2, b_2)$ iff $k_1\le k_2, a_1=a_2$ and $b_1=b_2$.

(iii) $(k, a, b) \sqsubseteq x $ iff  $a=x$ or  $kb\le x < b$.

If $u=(h, a, b)\in P_0$, then $\dr\!u=\{k, a, b): k\leq h\}$ is a chain. If $u=x\in P_0$, then $u = \bigvee\{(k, x, x): 0<k<1\}$, where each $(k, x, x)$ is a down-linear element and $\{(k, x, x): 0<k<1\}$ is a chain. Thus $P$ satisfies (2) of Corollary \ref{tm-c1-c2-faithful}.

Let $F$ be an irreducible nonempty Scott closed set of $P$ with an upper bound $v$. If $v=(h, a, b)\in P_0$, then $F\subseteq \dr (h, a, b)=\{(k, a, b): k\leq h\}$. Take $m=\bigvee\{k: (k, a, b)\in F\}$. Then $F=\dr (m, a, b)$, is the closure of  point $(m, a, b)$.

Now assume that $F$ does not have an upper bound in $P_0$, then $v=x$ for some $x\in P_0$. If $v\not\in F$, then due to the irreducibility of $F$, there exist $a, b$ such that $F\subseteq \{(k,a, b): 0<k<1\}$, which will imply that $F$ has an upper bound of the form $(m, a, b)$, contradicting the assumption. Therefore $v\in F$, implying that $F=\dr v$ (note that $F=\dr F$ is a lower set) is the  closure of point $v$.

It thus follows that $P$ satisfies (1) as well. By Corollary \ref{tm-c1-c2-faithful}, $P$ is SCL-faithful.

\end{example}


In \cite{Ho-Zhao}, Ho and Zhao introduced the following notions.
\begin{definition}\label{dn-beneath-c-compact}
  Let $L$ be a poset and $x, y \in L$. The element  $x$ is {\em beneath} $y$,
denoted by $x\prec y$, if for every nonempty Scott-closed set $S\subseteq L$ with
$\bigvee S$ existing, $y\le \bigvee S$   implies  $x\in S$.
An element $x$ of $L$ is called {\em C-compact} if $x\prec x$.
Let $\kappa(L)$  denote the set of all the C-compact elements of $L$.
\end{definition}

Let $P$ be a poset, $A\subseteq P$ finite.  The set $mub(A)$ of the minimal upper bounds of $A$ is complete, if for any upper bound $x$ of $A$, there exists $y\in mub(A)$ such that $y\le x$.

A poset $P$ is said to satisfy property $m$, if for all finite set $A\subseteq P$, $mub(A)$ is complete.

A poset $P$ is said to satisfy property $M$, if $P$ satisfies property $m$ and for all finite set $A\subseteq P$, $mub(A)$ is finite.

\begin{remark}
Let $L$ be a complete lattice and  $a\in L$ be a C-compact element. If  $x, y\in L$ such that $a\leq x\vee y$,
then $a\leq \bigvee (\downarrow\!x\cup\downarrow\!y)$ and $\downarrow\!x\cup\downarrow\!y$ is Scott closed, so $a\in \downarrow\!x\cup\downarrow\!y$, implying $a\le x $ or $a\le y$. Thus $a$ is $\vee$-irreducible.
\end{remark}

\begin{cor}
For any  dcpo $P$, $\kappa(C_{\sigma}(P))\subseteq Irr_{\sigma}(P)$. That is C-compact closet sets are all irreducible.
\end{cor}

\begin{lem} {\em \cite{He-Xu}}\label{lm-M-c-compact-ideal} Let $P$ be a dcpo. Then

{\em (1)} For all $x\in P$, $\downarrow\!x\in \kappa(C_{\sigma}(P))$.

{\em (2)} If $P$ satisfies property $M$, then  $A\in\kappa(C_{\sigma}(P))$ iff $A= \downarrow\!x$
for some $x\in P$.
\end{lem}

The following theorem gives the third  class of SCL-faithful dcpos using property M.

\begin{thm}
If $P$ is a dcpo satisfying property $M$ and the condition (2) in Theorem \ref{quasicontinuous elements}, then $P$ is SCL-faithful
\end{thm}
\begin{proof}
Let $P$ be a dcpo satisfying condition (2) in Theorem \ref{quasicontinuous elements} and property $M$, and $Q$ be a dcpo with an order isomorphism $F: C_{\sigma}(P)\rightarrow C_{\sigma}(Q)$.

 Then the restrictions $F: \kappa(C_{\sigma}(P))\rightarrow\kappa(C_{\sigma}(Q))$ and $F: Irr_{\sigma}(P)\rightarrow Irr_{\sigma}(Q)$ are all order isomorphisms.

For each $q\in Q$, by Lemma \ref{lm-M-c-compact-ideal}(1),  $\downarrow\!q\in \kappa(C_{\sigma}(Q))$,  then $F^{-1}(\downarrow\!q)={\downarrow\!x_q}$ for a unique $x\in P$ by Lemma \ref{lm-M-c-compact-ideal}(2).  Now  define a map $g: Q\to P$ such that $g(q)=x_q$ iff $F^{-1}({\downarrow\!q})={\downarrow\!x_q}$. The mapping $g$ satisfies the condition that $g(q_1)\leq g(q_2)$ iff $q_1\le q_2$   since $F^{-1}$ is an isomorphism.  Note that $\kappa(C_{\sigma}(Q))\cong \kappa(C_{\sigma}(P))\cong P$ is a dcpo.

Since $P$ satisfies the condition (2) in Theorem \ref{quasicontinuous elements},  by the proof of Theorem \ref{quasicontinuous elements} there is a monotone mapping $f: P\lra Q$ such that $F(\dr\!\! x)=\dr\!\! f(x)$ holds for every $x\in P$ (note that parts (1) and (2) of proof of Theorem\ref{quasicontinuous elements} do not need the condition that $P$ is bounded sober).

Then for any $x\in P$, $\dr x=F^{-1}(\dr f(x))$, so $x=g(f(x))$. Thus $g: Q\lra P$ is also a surjective mapping, therefore  an  isomorphism between $P$ and $Q$, as desired.
\end{proof}

 Note that  Kou's and Johnstone's examples of dcpos are bounded sober but  do not have property $M$.

\section{Remarks and some possible further work}
We close the paper with some extra remarks and problems for further exploration.

\begin{remark}

(1)  Recently, Ho, Jung and Xi \cite{HJX-16}  constructed a pair of non-isomorphic dcpos having isomorphic Scott topologies,  showing the existence of non-SCL-faithful dcpos. Their  counterexample also reveals that sobriety is not a sufficient condition for a dcpo to be SCL-faithful.

(2) If $P$ is an SCL-faithful dcpo and $P^*$ is the dcpo obtained by adding a top element to $P$, then one can show that $P^*$ is also SCL-faithful.
Let $X$ be the dcpo of Johnstone. Then $X^*$ is SCL-faithful, but $X^*$ is not bounded sober ($X$ is an irreducible Scott closed set of $X^*$ which is not the closure of any point of $X^*$). Thus a SCL-faithful dcpo  need not be bounded sober. So, bounded sobriety is not a necessary condition for a dcpo to be  SCL-faithful.

(3) The bounded sobriety in Theorem \ref{quasicontinuous elements} might be further weakened. We can try other conditions which are weaker than this, like the   ``dominatedness" used in \cite{HJX-16}.

(4) Given a class $\mathcal{M}$ of dcpos, define
$${\mathcal{M}}^{\flat}=\{P: P \mbox{ is a dcpo and for any } Q\in\mathcal{M}, C_{\sigma}(P)\cong C_{\sigma}(Q) \mbox{ implies } P\cong Q\}.$$

A class $\mathcal{M}$ of dcpos is called reflexive if ${\mathcal{M}}^{\flat\flat}=\mathcal{M}$.

 The class $\mathcal{S}$ of all SCL-faithful dcpos and the class $DCPO$ of all dcpos are reflexive. Do we have other reflexive classes of dcpos other than these two?

\end{remark}
\vskip 1cm
\noindent{\bf Acknowledgement}
The second author was supported by the NSF of China (61472343, 61300153) and University Science Research Project of Jiangsu Province (15KJD110006).

\end{document}